\documentclass[11pt]{article}
\usepackage{amsmath,amsfonts,amsthm}
\usepackage{subeqnarray}
\usepackage{stmaryrd}
\def\al{\alpha}
\def\be{\beta}

\def\vfi{\varphi}

    \def\si{\sigma}

\newtheorem{proposition}{Proposition}[section]
\newtheorem{definition}[proposition]{Definition}
\newtheorem{lemma}[proposition]{Lemma}

\newtheorem{corollary}[proposition]{Corollary}

\theoremstyle{definition}
\newtheorem{remark}{Remark}

\newdimen\AAdi%
\newbox\AAbo%
\def\AArm{\fam0 }
\def\AAk#1#2{\setbox\AAbo=\hbox{#2}\AAdi=\wd\AAbo\kern#1\AAdi{}}%
\def\AAr#1#2#3{\setbox\AAbo=\hbox{#2}\AAdi=\ht\AAbo\raise#1\AAdi\hbox{#3}}%
\def\BBone{{\AArm 1\AAk{-.8}{I}I}}%

\newcommand {\CL}{{\cal L}}

\newcommand{\llb}{\llbracket}
\newcommand{\rrb}{\rrbracket}

\newcommand{\disp}{\displaystyle}
\newcommand{\eps}{\varepsilon}
\newcommand{\8}{\infty}

\def\m1{{-1}}

\def\S{\Sigma}
\def\s{\sigma}

\newcommand{\N}{\mathbb{N}}
\newcommand{\R}{\mathbb{R}}

\newcommand{\0}{0^\8}
\newcommand{\1}{1^\8}
\newcommand{\bmeia}{\frac\beta2}
\newcommand{\tbmeia}{\frac{3\beta}{2}}

\title{Selection of measures for a potential with
two maxima at the zero temperature limit}
\author{A. T. Baraviera\thanks{baravi@mat.ufrgs.br, Instituto de Matem\'atica - UFRGS - Partially supported by DynEuroBraz},
R. Leplaideur\thanks{Renaud.Leplaideur@univ-brest.fr, D\'epartement de Math\'ematiques - Universit\'e de Brest- Partially supported by DynEuroBraz and  Convenio Brasil-Franca},
A. O. Lopes\thanks{arturoscar.lopes@gmail.com, Instituto de Matem\'atica - UFRGS - Partially supported by DynEuroBraz, CNPq, PRONEX -- Sistemas
Dinamicos,  INCT, Convenio Brasil-Franca, and beneficiary of CAPES financial support}}

\begin{document}

\maketitle

\begin{abstract}
For the subshift of finite type $\S=\{0,1,2\}^{\N}$ we study the convergence at temperature zero of the  Gibbs measure associated to a non-locally constant H\"older potential which admits only two maximizing measures. These measures are Dirac measures at two different fixed points.
The potential is flattest at one of these two fixed points.

The question we are interested is: which of these probabilities the invariant Gibbs state will select when temperature goes to zero?

We prove that on the one hand the Gibbs measure converges,  and at the other hand it does not necessarily converge to the measures where the potential is the flattest.

We consider a family of potentials of the above form; for some of them there is the selection of
a convex combination of the two Dirac measures, and for others there is a selection of the Dirac measure associated to the flattest point.
In the first case this is contrary to what was expected if we consider
the analogous problem in Aubry-Mather theory \cite{Anantharaman-flat}.
\vspace{0.5cm}

\noindent
{\bf Keywords:} selection of measures, transfer operator, Gibbs measures, ergodic optmization
\end{abstract}


\section{Introduction}
\subsection{optimization and selection}
The topic of optimization in Ergodic Theory deals with the study of maximizing or minimizing measures. Considering a dynamical system $(X,T)$ and $A:X\rightarrow \R$, a $A$-maximizing measure is a $T$-invariant probability measure $\mu$ such that
$$\int A\,d\mu=\max_{\nu\ T-inv}\left\{\int A\,d\nu\right\}.$$
Existence of maximizing measures is for instance ensured when $X$ is compact, and $T$ and $A$ are continuous.

The problem of selection deals with the limit at temperature zero of equilibrium state. A measure $\mu$ is an equilibrium state for $A$ if it satisfies
$$h_{\mu}(T)+\int A\,d\mu=\sup_{\nu\ T-inv}\left\{ h_{\nu}(T)+\int A\,d\nu\right\},$$
where $h_{\nu}(T)$ is the usual Kolmogorov entropy. It is well-known that any accumulation point for the equilibrium state $\mu_{\beta}$ associated to $\beta A$, where $\beta$ is a large positive real parameter, as $\beta$ goes to $+\8$ is a $A$-maximizing measure. In Statistical Mechanics the parameter $\beta$ is the inverse of the temperature. The study of selection is to consider the following question:  which maximizing measure is obtained as the limit of the equilibrium state associated to $\beta A$, when $\beta \to \infty$? In some cases there is no convergence (see \cite {chazottes-cv}). When the maximizing probability is unique there is convergence. Therefore, the interesting situation to analyze is when there is more than one $A$-maximizing probability.

\medskip

In \cite{Anantharaman-flat}, Anantharaman \emph{ and al.} study one example of selection for Lagrangian dynamics. They consider an external parameter $\epsilon$, and for each $\epsilon$ there is a natural probability which can be associated to an eigen-function problem.
There, they show, among other things,  that if the potential has only two points of maxima,
then this natural probability converges, when $\epsilon \to 0,$  to the Dirac measure concentrated in the point (which is maxima of the potential) were the
potential is the flattest.

In the present paper we study the same kind of problem but for the dynamics of the shift with three symbols.
The main difference between these two problems (Euler-Lagrange flow and the shift) is that for the case of the dynamics  of the shift,
every choice of $A$ is possible and makes sense  to be analyzed.  In Aubry-Mather theory
the dynamics (the Euler-Lagrange flow) depends of the Lagrangian
(or, potential) considered. In our case there is no relation of the potential with the dynamics. In \cite{Anantharaman-flat} the parameter $\epsilon$, such that $\frac{1}{\epsilon}$
goes to infinity, is related to viscosity solutions, and here
the parameter $\beta$ is the inverse of temperature. The question of selection
makes sense in both settings.
There is a natural  hope, that every result in one theory has its dual version for the other theory. This was the first motivation for this paper:
considering in $\S:=\{0,1,2\}^\N$ the  Holder potential
$$
    A(x) = \left\{
              \begin{array}{cc}
                -d(x, \0) & \text{if $x \in[0]$}   \\
               -3 d(x, \1) & \text{if $x \in [1]$}   \\
                -\al    &  \text{otherwise}
              \end{array}.
           \right.
 $$

 It is reasonable to say that the potential $A$ is more flat at $0^\infty$.
We initially expected that  the \emph{Gibbs measure} for the potential $\beta A$ converges to the  $\delta_{\0}$, as the temperature goes to 0, and we wanted to study how the selection occurs.

 For our surprise, for the case $\alpha<1$, we find out that the Gibbs measure always converges, but not to $\delta_{\0}$; it can select in the limit
 another convex
 combination of $\delta_{\0}$ and $\delta_{\1}$.

Moreover, the convex combination is not continuous on $\alpha$. Nevertheless it takes quite surprising values as a function
of $\alpha$. For this reason we believe that it would be very difficult to establish a global and general
selection theory for the class of all subshifts of finite type (with finite alphabet) and any potential,
not only due to this unexpected selection behavior but also because even convergence does not always
occur, see \cite{chazottes-cv}.

The invariant probability is obtained by the junction of the eigen-function and the eigen-probability \cite{PP}. A curious phenomena that happens in our examples is that the eigen-measure and
the eigen-function (see section (\ref{subsec-notation}) for definitions)  have opposite behavior. When $\beta \to \infty$, the eigen-measure became  exponential bigger  around $0^\infty$, when compared to points  around $1^\infty$. For the eigen-function the opposite   happens. Therefore, we need a very fine analysis of the control of the invariant Gibbs state.

The terminology "selection" was borrowed from the
theory of viscosity solutions (see for instance \cite{Anantharaman-flat} for references).

\subsection{Statement of result}
We work here with a full shift $\S$ over the alphabet  $\{0, 1, 2\}$. Points in $\S$ are sequences $x=(x_0,x_1,\ldots)$ with $x_i\in \{0, 1, 2\}$.
We will consider the usual terminology and the usual topology in $\S$.
 Hence, we recall that a cylinder $[X_0,\ldots X_k]$ is the set of points  $x=(x_n)$ such that $x_i=X_i$ for every $i\in \llb 0,k\rrb$. We also recall that the distance between $x=(x_n)$ and $y=(y_n)$ is defined by
$$d(x,y)=\frac1{2^{\min\{n,\ x_n\neq y_n\}}}.$$
The two special points $\0$ and $\1$  respectively denote the points $(0,0,\ldots)$ and $(1,1,\ldots)$. They are fixed points for the shift $\s$ over $\S$.

As we said above, we consider over
 this shift the Lipschitz potential $A$ defined as follows:
 $$
    A(x) = \left\{
              \begin{array}{cc}
                -d(x, \0) & \text{if $x \in [0]$}   \\
               -3 d(x, \1) & \text{if $x \in [1]$}   \\
                -\al    &  \text{otherwise}
              \end{array}
           \right.
 $$
for some $\al > 0$. Then this potential is always non-positive. There are only two maximizing measures, respectively $\delta_{\0}$ and $\delta_{\1}$. We point out that the potential is flattest close to $\0$.

It is well-known (see \emph{e.g.} \cite{bowen}) that there exists a unique equilibrium state for $\beta A$ (for all $\beta\in \R$).
It is also a Gibbs measure (see also Subsection \ref{subsec-notation}).

Our main result is:

\medskip
\noindent
{\bf Theorem}
{\it Let $\mu_{\beta}$ be the unique Gibbs measure associated to $\beta A$. Let $\rho$ be the golden mean $\rho:=\disp\frac{1+\sqrt5}2$. Then
\begin{enumerate}
	\item for $\alpha>1$, $\mu_{\beta}$ converges to $\frac12(\delta_{\0}+\delta_{\1})$ as $\be$ goes to $+\8$,
	\item for $\al=1$, $\mu_{\beta}$ converges to $\frac1{1+\rho^2}(\rho^2\delta_{\0}+\delta_{\1})$ as $\be$ goes to $+\8$,
	\item for $\alpha<1$, $\mu_{\beta}$ converges to $\delta_{\0}$ as $\be$ goes to $+\8$.
\end{enumerate}}

\medskip
\noindent
As we already said it above, this result is surprising because it was expected that in every cases $\mu_\be$ would converge to $\delta_{\0}$.
Discontinuity of the limit measure as a function of $\al$ is of course less surprising. Nevertheless, the values which appear in function of $\al$, and in particular for $\alpha=1$, are quite surprising.

For $\al=0$, it is expected that $\mu_\be$ converges to $\delta_{2^\8}$ (the flattest one !). Then, we could have expected the exact inverse situation between $\al<1$ and $\al>1$: {\it for $\al<1$, $\mu_\be$ would converge to $\disp\frac12(\delta_{\0}+\delta_{\1})$ and for $\al>1$, $\mu_\be$ would converge to $\delta_{\0}$}, the measure $\disp \frac12(\delta_{\0}+\delta_{\1})$ being a kind of ``smooth'' transition with the limit case $\delta_{2^\8}$ for the case $\al=0$. It turns out that this is not the case.

On the other hand, if $\al$ goes to $+\8$, the system looks, in some sense, the full-shift with two symbols $\{0,1\}^\N$.
In that case, it is not so surprising that the limit measure is
$\disp\frac12(\delta_{\0}+\delta_{\1})$, whatever the slopes are. Indeed, for $\{0,1\}^\N$, every
$\mu_\be$ typical orbit is an alternation of strings of 0's and 1's. Following \cite{leplaideur-max},
the convex combination would be given by the costs between the two maximizing zones, $\delta_{\0}$ and $\delta_{\1}$.
Hence, every typical orbit sees the two symbols and this is an heuristic argument which in some
sense justifies
that $\mu_\be$ converges to $\disp\frac12(\delta_{\0}+\delta_{\1})$.

We emphasize that one simple generalization of our theorem would be to replace $-3d(x,\1)$ with
some $-\Gamma d(x,\1)$, with $\Gamma>1$.
In this case the same result holds and  our method can easily be adapted. Nevertheless,
the computation would be a little bit more complicate and the formulas less convenient to be used.

\subsection{More notations- plan of the proof}\label{subsec-notation}
If $y=(y_1,y_2,\ldots)$ is a point in $\S$ and if $a=0,1,2$, we denote by $ay$ the point $(a,y_1,y_2,\ldots)$ in $\S$.

We recall that
the main tool is the transfer operator defined as follows:
\begin{eqnarray*}
   \CL_{\be} \vfi(x) &=& \sum_{y \in \si^{-1}(x)} e^{\be A(y)} \vfi(y)\\
   &=&e^{-\beta d(0x,\0)}\vfi(0x)+e^{-\beta d(1x,\1)}\vfi(1x)+e^{-\al\be}\vfi(2x).
\end{eqnarray*}
where $\be$ is the inverse of the temperature.
It acts on continuous functions and its dual operator, denoted by $\CL^*_{\be}$, acts on  probability measures.
Most of the time we will omit the subscript $\be$.

We know that there exists some function $H$ and some probability
measure $\nu$ such that $\CL(H)= e^P H$ and $\CL^* (\nu) = e^P \nu$.
Then, the probability measure $d\mu = H d\nu$ is $\si-$invariant and is the unique equilibrium state.
It is the so called Gibbs
measure associated to $\beta A$.

Throughout, they will thus be referred to as the eigen-measure and the eigen-function.

\bigskip
The plan of the proof of the main result of the paper is the following:

In Section \ref{sec-subaction} we give the exponential asymptotics for the eigen-function
(obtaining what is called a calibrated subaction) and the pressure.

In Section \ref{sec-eigen-measure} we prove the convergence of the eigen-measure to $\delta_{\0}$. For this we give precise values for the $\nu$-measures of rings.

In Section \ref{sec-eigen-function} we compute the exact values of the eigen-function
on the same rings considered before in Section \ref{sec-eigen-measure}.

In Section \ref{sec-end-proof} we finish the proof of our Theorem.

\section{Exponential asymptotic for the pressure and the eigen-function}\label{sec-subaction}
We first recall a usual definition in that theory.

\begin{definition}
We say that $u:\Omega \to \mathbb{R}$
is a calibrated subaction for $A$ if for any  $y$ we have
$$
 u(y) =\sup_{\sigma(x)=y}\{A(x)+ u(x)-m(A)\}.
$$
\end{definition}

We denote by $V$ any accumulation point for
$\displaystyle\frac{1}{\beta} \log H_{\beta}
$ as $\beta$ goes to $+\8$.
It is clearly a calibrated subaction, see \cite{CLT}. If we add a constant to a calibrated subaction, it will be also a calibrated subaction.

We remind that the Peierls' barrier is given by
$$
 h(x,y) =
 \lim_{\epsilon\to 0}\limsup_n\,
  \{ \sum_{j=0}^{n-1} A(\sigma^j(z))-m(A), n\geq 0, \sigma^n(z)=y, d(z,x)<\epsilon\}.
$$
\begin{remark}\label{rem-peierl-negative}
We let the reader check that for every $x\neq 0^\8,1^\8$ both numbers $h(0^\8,x)$ and $h(1^\8,x)$ are negative.
\end{remark}

It is known that if $u$ is a calibrated subaction
then it satisfies
$$
  u(y) = \sup_{\mathbf x \in \Omega} [h(\mathbf x, y)+  u(\mathbf x) ],
$$
where $h$ is the Peierls barrier and $\Omega$ is the Aubry-set \cite{CLT} \cite{GL2} (Theorem 10).

In the present case the Aubry-set is the union of the two fixed points  $p= 0^\infty$ and $q=1^\infty$.
In this way, any calibrated subaction is determined by its values on $p$ and $q$.

\begin{lemma}\label{lem-distances-subac}
 The functions defined by $u_0(x)=-d(x,0^\8)$ and $u_1(x)=-3d(x,1^\8)$  are both calibrated subactions.
\end{lemma}
\begin{proof}
The proof is only done for $u_0$, the other case  being similar.
We consider $y\in\S$ and we want to prove

\begin{equation}\label{eq1-peierl}
-d(0^\8,y)=:u_0(y) =max\{A(0y)+ u_0(0y),A(1y)+ u_0(1y),A(2y)+ u_0(2y)\}.
\end{equation}
We set $y=(y_0,y_1,y_2,\ldots)$.
We first assume that $y_0\neq 0$. Note that both $A(1y)$ and $A(2y)$ are negative and $u_0(1y)=u_0(2y)=-1$. Hence $u_0(y)=-1$ is bigger than both terms
$A(1y)+ u_0(1y)$ and $A(2y)+ u_0(2y)$. Now $A(0y)=-\frac12$ and $u_0(y)=-\frac12$. Hence \eqref{eq1-peierl} holds in that case.
Assume now that $y$ belong to the cylinder $0^n\*_0$. Then $u_0(y)=\frac{-1}{2^n}$. Again, note that $u_0(y)$ is bigger than both terms
$A(1y)+ u_0(1y)$ and $A(2y)+ u_0(2y)$. We also get
$$\frac{-1}{2^n}=\frac{-1}{2^{n+1}}+\frac{-1}{2^{n+1}}=A(0y)+u_0(0y).$$
Hence, \eqref{eq1-peierl} holds in that case too.
\end{proof}

Using Lemma \ref{lem-distances-subac} we can get a more simple formulation for $V$.
\begin{lemma}\label{lem-formule-V}
  $$V(x) = \sup\{ [  V(0^\infty ) - d(0^\infty , x)],[  V(1^\infty )- 3\,d(1^\infty , x) ]   \}$$
\end{lemma}
\begin{proof}
This follows from the fact that such $V$ is calibrated and from the expression of the Peierls barrier. Indeed, we claim that we have
$$h(0^\infty,y) = u_0(y) \mbox{ and } h(1^\infty,y) = u_1(y).$$
Again, we only prove that we get $h(0^\8,x)=u_0(x)=-d(x,0^\8
)$, the other equality being similar.

Let $x=(x_0,x_1,\ldots)$ be in $\S$. We get
$$u_0(x)=\max(h(0^\8,x)+u_0(0^\8),h(1^\8,x)+u_0(1^\8))=
\max(h(0^\8,x),h(1^\8,x)-1).$$
Note that $u_0(x)\ge-1$ and by Remark \ref{rem-peierl-negative} the Peierls barriers are both negative. Hence we must get
$$u_0(x)=h(0^\8,x).$$
\end{proof}

Now, we use properties of the eigenfunction $H_\beta$ to obtain
some relations satisfied by $V$. A calibrated subaction, in the present situation,  is determined by its values $ 0^\infty$ and $1^\infty$. We just need the relative values of $V$ at these points.

\begin{proposition}\label{prop-VetP}
For $\alpha>1$, we get $V(1^\infty) = V(0^\infty) + 1$ and $\disp\lim_{\beta\rightarrow+\8}\frac1\beta\log P(\beta)=-2$.

For $0<\alpha\le1$, we get $V(1^\infty) = V(0^\infty) + \alpha$ and $\disp\lim_{\beta\rightarrow+\8}\frac1\beta\log P(\beta)=-(1+\alpha)$.
\end{proposition}
\begin{proof}

Up to the fact that we consider a sub-family we assume that $\disp\lim_{\beta\rightarrow+\8}\frac1\beta\log P(\beta)$ exists and is equal to real number $\gamma$.

From the equation for the eigenfunction we get the pair of equations
\begin{subeqnarray}
 (e^{P(\beta)} -1)\, H_\beta(0^\infty) &=&
  e^{-\alpha \beta} H_\beta(2) + e^{-\frac{3}{2}\, \beta}\, H_\beta(1\,0^\infty),\slabel{equ1-eigenfunc}\\
 (e^{P(\beta)} -1)\, H_\beta(1^\infty) &=&
  e^{-\alpha \beta} H_\beta(2) + e^{-\frac{1}{2}\, \beta}\, H_\beta(0\,1^\infty).\slabel{equ2-eigenfunc}
\end{subeqnarray}
Remember that by Lemma \ref{lem-distances-subac} we get
\begin{eqnarray*}
  V(1\0 ) &= &\max\{ [  V(0^\infty ) - 1],[  V(1^\infty )- \frac32\, ]   \},
\\
 V(2\,x_1\,x_2.. ) &=& \max\{ [  V(0^\infty ) - 1],
  [  V(1^\infty )- 1 ]   \},\\
  V(01^\infty ) &= &\max\{ [  V(0^\infty ) - \frac12],[  V(1^\infty )- 3\, ]   \}.
\end{eqnarray*}

Then, taking $\frac1\beta\log$ in Equation \eqref{equ1-eigenfunc} and making  $\beta$ go to $+\8$ we get
\begin{eqnarray}
 \gamma+ V(0^\infty)&=&
 \max \{ [V(0^\infty)-1-\alpha],[ V(1^\infty)-3-\alpha],
 [V(0^\infty) -1-\frac{3}{2} ],[V(1^\infty) -\frac{3}{2}-\frac{3}{2}]\}\nonumber\\
&=&
 \max \{ [V(0^\infty)-1-\alpha],[ V(1^\infty)-3-\alpha],
 [V(0^\infty) -\frac{5}{2} ],[V(1^\infty) -3]\}\nonumber\\
 &=&
 \max \{[V(0^\infty)-1-\alpha],
  [V(0^\infty) -\frac{5}{2} ],[V(1^\infty) -3]\}.\label{equ1-formulev}
\end{eqnarray}
Similarly with \eqref{equ2-eigenfunc} we finally get
\begin{equation}
\gamma+ V(1^\infty)=\max\{[V(0^\infty)-1],[V(1^\infty) -7/2],[V(\1)-3-\alpha]\}.
\end{equation}

\medskip
{\bf We first deal with the case $\alpha>1$}. We will show that $V(1^\infty) = V(0^\infty) + 1$. We divide the analysis in two cases:

1) if $\alpha>3/2$, then, we have to solve
\begin{subeqnarray}
\gamma+ V(0^\infty)&=&\max\{ [V(0^\infty) -\frac{5}{2} ],[V(1^\infty) -3]\},\slabel{equ1-a}\\
	\gamma+ V(1^\infty)&=&\max \{ [V(0^\infty) -1 ],[V(1^\infty) -7/2]\}.\slabel{equ1-b}
\end{subeqnarray}
Now, we show that this system of equation is solvable if and only if
$V(0^\infty) -\frac{5}{2}\le V(1^\infty) -3$ and $V(0^\infty) -1\ge V(1^\infty) -7/2$.

Suppose that $V(0^\infty) -\frac{5}{2}> V(1^\infty) -3$. Then, we get $\gamma+ V(0^\infty)=  V(0^\infty) -5/2$, which shows that we have $\gamma=-5/2$.
Thus, we must have $V(0^\infty) -1\ge V(1^\infty) -7/2$ (otherwise \eqref{equ1-b} would give $\gamma=-\frac72$), and we get
$$  V(0^\infty) -1 =\gamma+ V(1^\infty)=-5/2 + V(1^\infty). $$
From this follows that  $V(1^\infty) =3/2+ V(0^\infty)$. This yields
$$V(1^\infty)-3=(3/2 + V(0^\infty))-3=V(\0)-\frac32>V(\0)-\frac52,$$
which produces a contradiction.

Then, we have
\begin{equation}\label{equ1-calculVetP}
	\gamma+ V(0^\infty)=V(1^\infty) -3
\end{equation}
An important consequence is that we must get $\gamma\ge -\frac52$. If $V(0^\infty) -1\le V(1^\infty) -7/2$, then \eqref{equ1-b} shows that $\gamma$ is equal to $-\frac72$ which is impossible. Hence we must get

\begin{equation}\label{equ2-calculVetP}
\gamma+ V(1^\infty)=V(0^\infty) -1.
\end{equation}
Finally,  \eqref{equ1-calculVetP} and \eqref{equ2-calculVetP} yield $\gamma=-2$, and $V(1^\infty)=V(0^\infty) +1$.

2) The case $1<\alpha\le\frac32$. The proof is similar.
 It is explicitly reproduced here, but the reader can skip it in a first reading.

The new system to solve is
\begin{subeqnarray}
\gamma+ V(0^\infty)&=&\max\{ [V(0^\infty) -(1+\alpha) ],[V(1^\infty) -3]\},\slabel{equ2-a}\\
	\gamma+ V(1^\infty)&=&\max \{ [V(0^\infty) -1 ],[V(1^\infty) -7/2]\}.\slabel{equ2-b}
\end{subeqnarray}
Again, we show that this system of equation is solvable if, and only if,
$V(0^\infty) -(1+\alpha)\le V(1^\infty) -3$ and $V(0^\infty) -1\ge V(1^\infty) -7/2$.

Suppose that $V(0^\infty) -(1+\alpha)> V(1^\infty) -3$. Then, we get $\gamma+ V(0^\infty)=  V(0^\infty) -(1+\alpha)$, which shows that we have $\gamma=-(1+\al)>-\frac52$.
Thus, we must have $V(0^\infty) -1\ge V(1^\infty) -7/2$ (otherwise \eqref{equ2-b} would give $\gamma=-\frac72$), and we get
$$  V(0^\infty) -1 =\gamma+ V(1^\infty)=-(1+\al) + V(1^\infty). $$
From this follows that  $V(1^\infty) =\al+ V(0^\infty)$. This yields
$$V(1^\infty)-3=(\al + V(0^\infty))-3=V(\0)-2>V(\0)-\frac52,$$
which produces a contradiction.

Then, we have
\begin{equation}\label{equ3-calculVetP}
	\gamma+ V(0^\infty)=V(1^\infty) -3
\end{equation}
An important consequence is that we must get $\gamma\ge -(1+\alpha)>-\frac52$. If $V(0^\infty) -1\le V(1^\infty) -7/2$, then \eqref{equ2-b} shows that $\gamma$ is equal to $-\frac72$ which is impossible. Hence we must get

\begin{equation}\label{equ4-calculVetP}
\gamma+ V(1^\infty)=V(0^\infty) -1.
\end{equation}
Finally,  \eqref{equ3-calculVetP} and \eqref{equ4-calculVetP} yield $\gamma=-2$, and $V(1^\infty)=V(0^\infty) +1$.

We point out here that the above discussion can be done for every sub-family of $\beta$'s.
In particular, this shows that $\disp\frac1\beta\log P(\beta)$ can have only one
accumulation point. In other words, it converges to $\gamma=-2$.

\medskip
{\bf Now, we deal with the case $\al\le 1$}. We will show that  $V(1^\infty) = V(0^\infty) + \alpha$.
The system we have to solve is
\begin{subeqnarray}
\gamma+ V(0^\infty)&=&\max\{ [V(0^\infty) -(1+\alpha) ],[V(1^\infty) -3]\},\slabel{equ3-a}\\
	\gamma+ V(1^\infty)&=&\max \{ [V(0^\infty) -1 ],[V(1^\infty) -7/2],[V(1^\infty) -3-\al]\}.\slabel{equ3-b}
\end{subeqnarray}
We show that, whatever is the case $\al\le\frac12$ or $\al\ge\frac12$, the system can be solved if, and only if,
$\disp V(0^\infty) -(1+\alpha) \ge V(1^\infty) -3$ and $V(0^\infty) -1 \ge V(1^\infty) -7/2,V(1^\infty) -3-\al$.

Let us proceed by contradiction and assume we get $\disp V(0^\infty) -(1+\alpha) < V(1^\infty) -3$. In that case, if we assume that we get $V(0^\infty) -1 \ge V(1^\infty) -7/2,V(1^\infty) -3-\al$, then the system to solve is exactly given by equations \eqref{equ1-calculVetP} and \eqref{equ2-calculVetP}. This yields $\gamma=-2$, and $V(1^\infty)=V(0^\infty) +1$.

Then, we get $\disp V(\1)-3=V(\0)-2\le V(\0)-(1+\alpha)$ which produced a contradiction with our assumption $\disp V(0^\infty) -(1+\alpha) < V(1^\infty) -3$.

This means that $V(0^\infty) -1 \le V(1^\infty) -7/2,V(1^\infty) -3-\al$, and the bigger  term only depends on the relative position of $\al$ with respect to $\frac12$. Depending of this position, we get $\disp\gamma=-\frac72$ or $\gamma=\disp-3-\alpha$. Then \eqref{equ3-a} would give in both case
$$V(\0)-\gamma>V(\0)-(1+\al),$$
which produces a contradiction. Hence, we must get $\disp V(0^\infty) -(1+\alpha) \ge V(1^\infty) -3$ and
\begin{equation}\label{equ1-petital}
\gamma=-(1+\al).
\end{equation}
If $V(0^\infty) -1 \ge V(1^\infty) -7/2,V(1^\infty) -3-\al$ does not hold, then we would get $\disp\gamma=-\frac72$ or $\gamma=\disp-3-\alpha$, which is impossible. Thus we must get $V(0^\infty) -1 \ge V(1^\infty) -7/2,V(1^\infty) -3-\al$ and we finally get
\begin{equation}\label{equ2-petital}
V(\1)+\gamma=V(\1)-(1+\al)=V(\0)-1.
\end{equation}
This finishes the proof of the proposition (again $\gamma$ is  the unique possible accumulation point for $\disp\frac1\beta\log P(\beta)$).

\end{proof}

\section{The eigen-measure $\nu$}\label{sec-eigen-measure}
In this section we study the eigen-measure $\nu_{\beta A}$. We prove it converges to the Dirac measure
$\delta_{\0}$. We also show exact limit values on special sets.

\subsection{A useful function}
We define and study a function $F$ depending on the pressure $P(\beta)$ and on the parameter $\beta$.
\begin{definition}\label{def-FPbeta}
 For $Z \geq 0$ and $\be \geq 0$
$\disp
    F(Z, \be) := \sum_{k = 0 }^{\infty} e^{-kZ} e^{\frac{\be}{2^{k+1}}}$
 and its partial sum
 $\disp F_n(Z,\be) := \sum_{k = 0 }^{n} e^{-kZ} e^{\frac{\be}{2^{k+1}}}
 $.
\end{definition}
 Clearly, $F_n(Z, \be) \to F(Z, \be)$ when $n \to \infty$.

 We remind that as $\beta$ goes to $+\8$, $P$ goes exponentially fast to $0$.
The asymptotic behavior of $F$ (for $\be$ very large) can be obtained as
follows:
\begin{lemma}\label{lem-FZbeta}
For every $\beta>2\ln 2$ we get
$$
  \left|F(P,\beta)-\frac{1}{P}\right|\le  \frac{\beta e^{\beta/2} }{2\, \ln 2} (2 + \sum_{n\geq 1} (\frac{P}{\ln 2})^n).
$$
\end{lemma}
\begin{proof}
Let us consider a positive $Z$.
Note that the function $x\mapsto \disp -Zx+\frac{\beta}{2.2^{x}}$ is decreasing on $\R_{+}$. We can thus compare the sum and the integral:
$$\int_{0}^{+\8}Ze^{-xZ}e^{\bmeia\frac1{2^{x}}}\,dx\le ZF(Z,\beta)\le \int_{0}^{+\8}Ze^{-xZ}e^{\bmeia\frac1{2^{x}}}\,dx+Ze^{\frac\beta2}.$$
Let us study the integral. We get
\begin{eqnarray*}
\int_{0}^{+\8}Ze^{-xZ}e^{\bmeia\frac1{2^{x}}}\,dx&=&\left[-e^{-xZ}e^{\bmeia\frac1{2^{x}}}\right]^{+\8}_{0}-\int_{0}^{+\8} \frac\beta2 e^{-xZ}\frac{\ln 2}{2^{x}}e^{\bmeia\frac1{2^{x}}}\,dx.\\
&=&e^{\bmeia}-\int_{0}^{+\8} \frac\beta2 e^{-xZ}\frac{\ln 2}{2^{x}}e^{\bmeia\frac1{2^{x}}}\,dx.
\end{eqnarray*}
Let us set $u=\frac1{2^{x}}$ in this last integral. We get
$$\int_{0}^{+\8}Ze^{-xZ}e^{\bmeia\frac1{2^{x}}}\,dx=e^{\bmeia}-\int_{0}^1 \frac\beta2 e^{-Z\frac{\ln u}{\ln 2}}e^{\frac\beta2 u}\,du.$$
Writing $\disp e^{-Z\frac{\ln u}{\ln 2}}=\sum_{n=0}^{+\8}\frac1{n!}\left(-Z\frac{\ln u}{\ln 2}\right)^n$  we get
$$\int_{0}^{+\8}Ze^{-xZ}e^{\bmeia\frac1{2^{x}}}\,dx=e^{\bmeia}-\int_{0}^1 \frac\beta2 \sum_{n=0}^{+\8}\frac1{n!}\left(-Z\frac{\ln u}{\ln 2}\right)^ne^{\frac\beta2 u}\,du.$$
To get the inverse of the two sums we remind that
$\disp\int_{0}^1|\ln u|^n\,du=\int_{0}^{+\8}v^ne^{-v}\,dv=n!$. Then for $Z<\ln 2$ we get

\begin{eqnarray*}
\int_{0}^{+\8}Ze^{-xZ}e^{\bmeia\frac1{2^{x}}}\,dx&=&e^{\bmeia}- \sum_{n=0}^{+\8}\frac1{n!}\left(\frac{-Z}{\ln 2}\right)^n\int_{0}^1 \frac\beta2 ({\ln u})^ne^{\frac\beta2 u}\,du\\
&=& 1-\sum_{n=1}^{+\8}\frac1{n!}\left(\frac{-Z}{\ln 2}\right)^n\int_{0}^1 \frac\beta2 ({\ln u})^ne^{\frac\beta2 u}\,du.
\end{eqnarray*}
Now, note that
$$\disp\left| \frac1{n!}\left(\frac{-Z}{\ln 2}\right)^n\int_{0}^1 \frac\beta2 ({\ln u})^ne^{\frac\beta2 u}\,du\right|\le\frac1{n!}\left(\frac{Z}{\ln 2}\right)^n\frac\beta2 e^{\frac\beta2}\int_{0}^1 |\ln u|^n\,du=\left(\frac{Z}{\ln 2}\right)^n\frac\beta2 e^{\frac\beta2}.$$
We also recall that for positive $\beta$, the pressure is strictly smaller than the topological entropy $\ln 2$.
This shows the lemma.
\end{proof}

\subsection{The eigen-measure on the cylinders $[0]$ and $[1]$}
We remind that  the eigen-probability for $\beta A$,  $\nu_{\beta}$, is a conformal measure:
for any cylinder set $B$ where $\sigma$ is injective
$$
 \nu_{\beta} (\sigma(B)) = \int_B e^{P(\beta)- \beta A(x)} d \, \nu_{\beta} (x).
$$
We shall use this simple relation to compute exact values for $\nu_{\beta}$ of some special cylinders.

For simplicity we drop the subscribe $\beta$ in $\nu_{\beta}$ and simply write $\nu$.
We shall also use the notation $*_0$ for the pair of symbols which are not $0$ and
$*_1$ for the pair of symbols which are not $1$.
Then
$$
 [0*_0]= [01] \sqcup [02] \qquad \text{and} \qquad [1*_1] = [10]\sqcup [12]
$$
(and the unions are disjoint).

We can now estimate the measures of the cylinders $[0]$ and $[1]$.
\begin{lemma}\label{lem-eigen cylindre-couronne}
$$
 \nu[0] = e^{-\frac{\beta}{2}}\, F(P,\beta) \,\nu [0*_0 ]
$$
$$
 \nu[1] = e^{-\frac{3 \beta}{2}}\, F(P,3 \beta) \,\nu [1*_1 ]
$$
\end{lemma}
\begin{proof}
Conformality yields
$$
 \nu[0*_0]= \nu[\sigma( 00*_0)]=e^{P+ \frac{\beta}{2^2}} \, \nu[00*_0]=
 e^{2P+  \frac{\beta}{2^3}+\frac{\beta}{2^3}} \, \nu[000*_0],
$$
and so on. By induction we get
\begin{equation}\label{eq1a}
\nu[0*_0]=e^{(n-1)\,P+ \beta \, (\frac{1}{2^2}+...+ \frac{1}{2^n})}\, \nu [\underbrace{00 \ldots 0}_n*_0 ].
\end{equation}

Hence, we get
$$
  \nu[0]= \sum_{n=1}^\infty  \nu [ \underbrace{00 \ldots 0}_n *_0]=
 \sum_{n=1}^\infty e^{-(n-1) \,P\,} \, e^{-\frac{\beta}{2}} \, e^{\frac{\beta}{2^n}}\,\nu [0*_0 ]=  \, e^{-\frac{\beta}{2}}\, F(P,\beta) \,\nu [0*_0 ].
$$

Similarly we get
$\disp
 \nu[1] = \, e^{-\frac{3\, \beta}{2}}\, F(P,3\,\beta) \,\nu [1*_1 ]$.
\end{proof}

Using $[0*_0]= [01] \sqcup [02]$ and$[1*_1] = [10]\sqcup [12]
$ and the conformal property of $\nu$ we obtain the following system:

\begin{subeqnarray}\label{equ-systeme1-eigenmesure}
 \nu [1 *_1 ]&=& \nu[2] \, e^{-P - \frac{3 \beta}{2}} + \nu[0]e^{-P - \frac{3 \beta}{2}}.\slabel{equ1-systm1-eigenmeasure}\\
  \nu [0 *_0 ]&=& \nu[2] \, e^{-P - \frac{ \beta}{2}} + \nu[1]e^{-P - \frac{ \beta}{2}}.\slabel{equ2-systm1-eigenmeasure}
\end{subeqnarray}
This system is the key point to determine the convergence of the eigen-measure.

\begin{proposition}\label{prop-cv-eigenmeas}
The ratio $\disp\frac{ \nu[0 ]}{  \nu[1 ]}$ goes exponentially fast to $+\8$ as $\beta$ goes to $+\8$.
  \end{proposition}
\begin{proof}
 By Lemma \ref{lem-eigen cylindre-couronne} the system \eqref{equ-systeme1-eigenmesure} can be transformed into a system in $\nu[0]$, $\nu[1]$, and  $\nu[2]$:
\begin{eqnarray*}
 \nu[0] &=&
  e^{-\beta/2}\, F(P,\beta)\,\{\nu[2] \,
  e^{-P - \frac{ \beta}{2}} + \nu[1]e^{-P - \frac{ \beta}{2}}\,\} \\
 \nu[1] &=&
  e^{-(3\, \beta)/2}\, F(P,3\beta)\,
  \{\nu[2] \, e^{-P - \frac{3 \beta}{2}} + \nu[0]e^{-P - \frac{3 \beta}{2}} \,\}
\end{eqnarray*}
This yields

 \begin{equation}
\label{equ1-rapport-nu-cylindres}
 \frac{ \nu[0 ]}{  \nu[1 ]}=e^{2 \beta} \,
  \frac{F(P,\beta) \,( \,1  +  e^{-P-3\,\beta}\, F(P,3\beta)\, )}
{ F(P, 3 \beta)\, (\, 1 + e^{-P-\beta}\, F(P,\beta)\,)  }
\end{equation}

Finally, when $\beta \to \infty$, Proposition \ref{prop-VetP} and  Lemma \ref{lem-FZbeta} show that
$ \disp \frac{ \nu[0 ]}{  \nu[1 ]}$ goes to $+\8$ exponentially fast. The exponential speed is larger than $1-\eps$ for every positive $\eps$.
\end{proof}

We point out that Lemma \ref{lem-eigen cylindre-couronne} also allow to transform  the system \eqref{equ-systeme1-eigenmesure}  into a system in $\nu([0*_{0}])$, $\nu([1*_{1}])$, and  $\nu(2)$. From this system we get

\begin{equation}
\label{equ2-syst-estrelas}
 \frac{ \nu[0 *_0]}{  \nu[1 *_1]}= e^{\, \beta} \,
  \frac{ \, ( \,1 +  e^{-P-3\beta}\, F(P,3\,\beta)\, )}
   { (\, 1 + e^{-P-\beta}\, F(P,\beta)\,)  } .
\end{equation}
Nevertheless, at that point of the proof we do not have enough information on $P$ to conclude which is the limit of the ratio. Proposition \ref{prop-VetP} and Lemma \ref{lem-FZbeta} just allow to ensure that $\disp\frac1\beta\log  \frac{ \nu[0 *_0]}{  \nu[1 *_1]}$  goes to 0.
However, we can get ratios for  other rings:

\begin{corollary}\label{cor-ratio-eigenmes-couronnes}
For every $n\ge 2$,
$$\disp
 \frac{ \nu[0^n *_0]}{  \nu[1^n *_1]}=e^{ \beta(1-\frac{1}{2^{n-1}})} \frac{ \,\nu[ 0*_0]}{\,  \nu[1 *_1]}.$$
The ratio $\disp \frac{ \nu[0^n *_0]}{  \nu[1^n *_1]}$ goes to $+\8$ as $\beta$ goes to $+\8$ with exponential speed larger than $(1-\frac{1}{2^{n-1}})-\eps$ for every positive $\eps$.
\end{corollary}

\subsection{Convergence of the eigen-measure}
In this subsection we get a finer estimate for $P(\beta)$ and conclude that $\nu$ goes to the Dirac measure $\delta_{\0}$.

The conformal property yields

  \begin{equation}
\label{equ1-nu2}
 \nu([2]) = \nu([20])+\nu([21]) +\nu([22]) = e^{-P-\al\be}(\nu[0]+\nu[1]+\nu[2] ) =
    e^{-P-\al \be} .
\end{equation}
On the other hand the solution of the system obtained in the proof of Proposition \ref{prop-cv-eigenmeas} shows that we have
\begin{eqnarray*}
\nu([0])&=& \frac{1+e^{-P-3\beta}F(P,3\beta)}{1-e^{-2P}F(P,\beta)F(P,3\beta)e^{-4\beta}}F(P,\beta)e^{-P-\beta}\nu([2]),\\
\nu([1])&=& \frac{1+e^{-P-\beta}F(P,\beta)}{1-e^{-2P}F(P,\beta)F(P,3\beta)e^{-4\beta}}F(P,3\beta)e^{-P-3\beta}\nu([2]).
\end{eqnarray*}
Using the formula $\nu([0])+\nu([1])+\nu([2])=1$ we get another expression for $\nu([2])$:
\begin{eqnarray}
&&1=\nu([2])\left(1+ \frac{1+e^{-P-3\beta}F(P,3\beta)}{1-e^{-2P}F(P,\beta)F(P,3\beta)e^{-4\beta}}F(P,\beta)e^{-P-\beta}+\right.\nonumber\\
&&\hskip 3cm\left. \frac{1+e^{-P-\beta}F(P,\beta)}{1-e^{-2P}F(P,\beta)F(P,3\beta)e^{-4\beta}}F(P,3\beta)e^{-P-3\beta}\right)\nonumber\\
&&=\nu([2])\left( \frac{1+e^{-P-\beta}F(P,\beta)+e^{-P-3\beta}F(P,3\beta)+e^{-2P-4\beta}F(P,\beta)F(P,3\beta)}{1-e^{-2P}F(P,\beta)F(P,3\beta)e^{-4\beta}}\right).\label{equ2-nu2}
\end{eqnarray}
Lemma \ref{lem-FZbeta} and Proposition \ref{prop-VetP} show that whatever the value of $\alpha$ is, $e^{-P-3\beta}F(P,3\beta)$
 goes to $0$ as $\beta$ goes to $+\8$. On the other hand, $e^{-P-\beta}F(P,\beta)$ is exponentially big
 (of order $e^{\beta}$ if $\alpha$ is bigger than 1 and $e^{\alpha\beta}$ if $\alpha$ is smaller than 1).
Remember that Equation \eqref{equ1-nu2} shows that $\nu([2])$ goes exponentially fast to $0$ with exponential speed  $-\alpha\beta$.
\begin{lemma}
\label{lem-vitesse-fine-P}
If $\alpha> 1$, we get $\disp \lim_{\beta\rightarrow+\8}P(\beta)e^{2\beta}=1$. For $\alpha=1$, $P(\beta)e^{2\be}$ goes to $\frac{1+\sqrt5}2$.
\end{lemma}
\begin{proof}
We first do the case $\alpha>1$.
As we said above, the numerator in the right hand side of \eqref{equ2-nu2} has order $e^{\beta}$. On the other hand $\nu([2])$ has order $e^{-\al\be}$. Therefore, the denominator of the right hand side of  \eqref{equ2-nu2}  goes to 0 with exponential speed $e^{(1-\alpha)\beta}$.
Then, Lemma \ref{lem-FZbeta} shows that $P(\beta)e^{-2\beta}$ goes to $1$.

Let us now deal with the case $\alpha=1$. Copying what we did above we get
$$e^{P}=\frac{e^{-2\be}}P\frac{1+\eps_{1}(\beta)}{1-e^{-2P}\left(\frac{e^{-2\beta}}{P}\right)^2(1+\eps_{2}(\beta))},$$
with $\eps_{i}(\beta)$ going to 0 as $\beta$ goes to $+\8$. Let $l$ be any accumulation point for $\disp {P}e^{2\beta}$. We thus get
$$=\frac{\frac1l}{1-\frac1{l^2}}=\frac{l}{l^2-1}.$$
 This yields  $l=\frac{1+\sqrt5}2$.
\end{proof}

\begin{corollary}
\label{coro-lim-ratio-estrella}
As $\beta$ goes to $+\8$, the ratio $\disp\frac{ \nu[0 *_0]}{  \nu[1 *_1]}$ goes to 1 for $\alpha>1$ to $\frac{\sqrt5+1}2$ for $\alpha=1$ and to $+\8$ for $\alpha<1$. The convergence is non-exponential for $\al\ge 1$ and has exponential speed $1-\al$ if $\al<1$.
\end{corollary}
\begin{proof}
We remind that Equation \eqref{equ2-syst-estrelas} gives
$$ \frac{ \nu[0 *_0]}{  \nu[1 *_1]}= e^{\, \beta} \,
  \frac{ \, ( \,1 +  e^{-P-3\beta}\, F(P,3\,\beta)\, )}
   { (\, 1 + e^{-P-\beta}\, F(P,\beta)\,)  } .
$$
We already know that $e^{-3\beta}F(P,3\beta)$ goes to $0$ as $\beta$ goes to $+\8$. The denominator has for dominating term $\frac{e^{-\beta}}P$. For $\alpha<1$ we directly get that $\disp\frac{ \nu[0 *_0]}{  \nu[1 *_1]}$ goes to $+\8$. For $\alpha\ge 1$ we use Lemma \ref{lem-vitesse-fine-P}.
\end{proof}

Equation \ref{equ1-nu2} shows that $\nu([2])$ goes to $0$ as $\beta$ goes to $+\8$. Then Proposition \ref{prop-cv-eigenmeas} yields:

\begin{corollary}
\label{coro-cv-eigenmes}
The measure $\nu$ goes to the Dirac measure $\delta_{\0}$ as $\beta$ goes to $+\8$.
\end{corollary}

\section{The eigen-function $H$}\label{sec-eigen-function}

 In this section we get estimates at the non-exponential scale for the asymptotic behavior of the eigenfunction
 $H_{\be}$. In what follows, for simplicity, we will drop the
 subindex $\be$.

 \subsection{The exponential scale is not deterministic}

 We know  that
 \begin{equation}\label{equ-def-H}
     H(x) = \lim_{N \to \infty} \frac{1}{N} \sum_{k=0}^{N-1} \frac{\CL^k(\BBone)(x)}{e^{kP}}
\end{equation}
where $\CL$ is the transfer operator (see Subsection \ref{subsec-notation}).
 We recall that $*_0$ (resp. $*_1$) denotes any symbol different to 0 (resp. to $1$).
 We start with the following result.
\begin{lemma}\label{lem-H-loc const}
  The eigen-function is constant on cylinders $[0^n*_0]$, $[1^n*_1]$ and $[2]$.
 \end{lemma}
\begin{proof}
  Owing to Equation \ref{equ-def-H}, it is sufficient to prove that for every $k$, $\CL^k(\BBone)$ is constant on cylinders $[0^n*_0]$, $[1^n*_1]$ and $[2]$.
  For $x$ in $\S$, we get
  $$\CL^k_\beta(\BBone)(x)=\sum_{z\in\{0,1,2\}^k}e^{\be.S_k(A)(zx)},$$
  where $S_k(A)$ is the Birkhoff sum $A+A\circ\s+\ldots+A\circ^{k-1}\s$. Now, note that the potential
  is constant on the cylinders $[0^m*_0]$, $[1^m*_1]$  (whatever $m\ge 1$ is) and $[2]$. This finishes the proof of the lemma.
\end{proof}

We emphasize here that the information we get on the subaction (namely the exponential asymptotic for $H$) and on the eigen-measure  are not yet sufficient to conclude the proof. Indeed, one important fact is that the eigen-measure and
the eigen-function have opposite behavior:

\begin{lemma}\label{lem-sens-inverse}
For $\al\ge1$ and for every integer $n\ge 1$, $\disp \lim_{\be\to+\8}\frac1\be\log\frac{\mu([0^n*_0])}{\mu([1^n*_1])}=0$.
\end{lemma}
\begin{proof}
By definition we get $\disp\frac{\mu([0^n*_0])}{\mu([1^n*_1])}=\frac{H(0^n*_0)\nu([0^n*_0])}{H(1^n*_1)\nu([1^n*_1])}$. Using Corollaries \ref{cor-ratio-eigenmes-couronnes} and \ref{coro-lim-ratio-estrella}, we get that $\disp \frac1\be\log\frac{\nu([0^n*_0])}{\nu([1^n*_1])}$ goes to $\disp 1-\frac{1}{2^{n-1}}$ as $\be$ goes to $+\8$.

On the other hand, Lemma \ref{lem-formule-V} and Proposition \ref{prop-VetP} shows that $\disp \frac1\be\log\frac{H(0^n*_0)}{H(1^n*_1)}$ goes to $\disp-1+\frac{1}{2^{n-1}}$ as $\be$ goes to $+\8$.
\end{proof}

\begin{remark}
For $\al<1$ is also possible to show, following the same
procedure, that
$\disp \lim_{\be\to+\8}\frac1\be\log\frac{\mu([0^n*_0])}{\mu([1^n*_1])}=2-2 \al$.
\end{remark}

Lemma \ref{lem-sens-inverse} shows that the convergence and the study of selection for $\mu$ cannot be obtained
at the exponential scale. We thus must get more precise estimates.

\subsection{Estimation at the non-exponential scale}
We recall that the functions $F(P,\beta)$ and $F_{n}(P,\be)$ were defined in Definition \ref{def-FPbeta}.

\begin{lemma}\label{lem-formule-H}
For every $n\ge 1$ we get
\begin{eqnarray}
H(0^n*_0)&=&e^{(n-1)P-\frac\be{2^n}}\frac{(e^P-1)}{e^P+e^{-\al\be}}\left[e^{P+\be}H(\1)-\left(F_{n-2}(P,\beta)(1+e^{-P-\al\be})+e^{(1-\al)\be}\right)H(\0)\right]
,\nonumber\\
&&\label{equ-H0n}\\
H(1^n*_1)&=&e^{(n-1)P-\frac{3\be}{2^n}}\frac{(e^P-1)}{e^P+e^{-\al\be}}\left[e^{P+3\be}H(\0)-\left(F_{n-2}(P,3\beta)(1+e^{-P-\al\be})+e^{(3-\al)\be}\right)H(\1)\right],\nonumber\\
&&\label{equ-H1n}
\end{eqnarray}
where $F_{-1}\equiv 0$.
	\end{lemma}
\begin{proof}
Using the equality $\CL(H)=e^PH$ we get the following system of equations
\begin{equation}\label{equ-systeme-H}
\left\{
\begin{array}{rrrcl}
	 e^{-\bmeia}H(0*_0)&&+e^{-\al\be}H(2)&=&(e^P-1)H(\1),\\
	 &e^{-\tbmeia}H(1*_1)&+e^{-\al\be}H(2)&=&(e^P-1)H(\0),\\
	 e^{-\bmeia}H(0*_0)&+e^{-\tbmeia}H(1*_1)&+(e^{-\al\be}-e^P)H(2)&=&0.\\
\end{array}
\right.
\end{equation}
Solving this system in terms of $H(\1)$ and $H(\0)$ we find:
\begin{eqnarray}
H(0*_0)&=&e^{\bmeia}\frac{(e^P-1)}{e^P+e^{-\al\be}}\left[e^PH(\1)-e^{-\al\be}H(\0)\right]
\label{equ1-h0*}\\
H(1*_1)&=&e^{\tbmeia}\frac{(e^P-1)}{e^P+e^{-\al\be}}\left[e^PH(\0)-e^{-\al\be}H(\1)\right]
\label{equ1-h1*}
\end{eqnarray}
Again, the equality $\CL(H)=e^p H$ yields
$$e^PH(0^n*_0)=e^{-\frac\be{2^{n+1}}}H(0^{n+1}*_0)+e^{-\tbmeia}H(1*_1)+e^{-\al\be}H(2).$$
Introducing  the second equation in \eqref{equ-systeme-H}, we get
$$H(0^{n+1}*_0)=e^{P+\frac\be{2^{n+1}}}H(0^n*_0)+e^{\frac\be{2^{n+1}}}(e^P-1)H(\0).$$
By induction, we get for every $n\ge 2$ an expression of $H(0^n*_0)$ in function of $H(\0)$ and $H(0*_0)$. Then, introducing \eqref{equ1-h0*} in this expression, we let the reader check that we get \eqref{equ-H0n}. The proof of \eqref{equ-H1n} is similar.
\end{proof}
As we said above, the exponential scale is not sufficient to determine the limit
and the selection for the Gibbs measure. Due to the values of the subactions, the good
parameter to estimate is $\disp e^\be\frac{H(\0)}{H(\1)}$. Lemma \ref{lem-formule-H}
allows us to solve that problem.

\begin{proposition}\label{prop-cv-H0H1}
\begin{itemize} As $\be$ goes to $+\8$ we get the following limits:
\item[(i)]
if $\alpha>1$, then,
$\disp \lim_{\beta\to+\8}e^\beta \frac{H(\0)}{H(\1)}=1$,
\item[(ii)]
if $\alpha=1$, then,
$\disp \lim_{\beta\to+\8}e^\beta \frac{H(\0)}{H(\1)}=\frac{1+ \sqrt{5}}{2}$,
\item[(iii)]
if $0<\alpha<1$, then,
$\disp \lim_{\beta\to+\8}e^{\beta}  \frac{H(\0)}{H(\1)}=+\infty$.

\end{itemize}
\end{proposition}
\begin{proof}
Equalities \eqref{equ-H0n} and \eqref{equ-H1n} yield for any fixed $n$
\begin{equation}\label{eq1}
e^{\beta-\frac{\beta}{2^{n-1} }}\,\frac{H(0^n*_0) }{H(1^n*_1)}=\frac{e^P - [ \,F_{n-2} (P,\beta)\, (1+ e^{-P-\alpha\, \beta} )\,
 e^{-2 \beta}\,  + e^{-(1+\alpha) \beta}\,] \,(e^{\beta}\,  \frac{H(0) }{H(1)})  }
{  e^P \,(e^{\beta}\,  \frac{H(0) }{H(1)})\,- \,[ \,F_{n-2} (P,3\beta)\,
 (1+ e^{-P-\alpha\, \beta} )\, e^{-2 \beta}\,  + e^{(1-\alpha) \beta}\,]   }.
\end{equation}

For, $\beta$ fixed,  we set $x =x_\beta=e^{\beta}\,  \frac{H(0) }{H(1)}$. Then, taking the limit as $n$ goes to $+\infty$ we get
$$
 x=
\frac{e^P - [ \,F(P,\beta)\, (1+ e^{-P-\alpha\, \beta} )\,
 e^{-2 \beta}\,  + e^{-(1+\alpha) \beta}\,] \,x }
 { e^P \,x\,- \,[ \,F(P,3\beta)\, (1+ e^{-P-\alpha\, \beta} )\,
 e^{-2 \beta}\,  + e^{(1-\alpha) \beta}\,]   },
$$
(the eigen-function is continuous).
Let us set $a=d=e^P$ and
$$
 b= \,-\, [ \,F (P,\beta)\, (1+ e^{-P-\alpha\, \beta} )\,
   e^{-2 \beta}\,  + e^{-(1+\alpha) \beta}\,],
$$
$$
 c=\,-\, [ \,F (P,3\beta)\, (1+ e^{-P-\alpha\, \beta} )\,
  e^{-2 \beta}\,  + e^{(1-\alpha) \beta}\,].
$$

We can write the above equation in the form
$$
	x= \frac{a + b \,x}{d\,x+c}.
$$
 As $x$ is positive we can solve this equation and we get

\begin{equation}\label{equ-xabcd}
 x= \frac{(b-c)\, +\, \sqrt{ (c-b)^2 + 4\, a\, d}   }{ 2\,d     }.
\end{equation}

Note that
$$
 (b-c) = (\,F(P,3 \beta) -F(P, \beta)\,)\,\, e^{-2\, \beta} \,
 (1+ e^{-P-\alpha\,\beta} )+ e^{-\alpha \,\beta} \, (e^\beta- e^{-\beta}).
$$
Now, Lemma \ref{lem-FZbeta} shows that
$\disp e^{-\, 2\, \beta} \, ( \,F(P,3\, \beta) - F(P,\beta)\,)\to 0$
when $\beta$ goes to $+\infty$. On the other hand we get,
\begin{description}
	\item[] for $\alpha>1$,
$\disp e^{-\alpha \,\beta} \, (e^\beta- e^{-\beta})\to 0.
$
\item[] for $\alpha<1$,
$\disp e^{-\alpha \,\beta} \, (e^\beta- e^{-\beta})\to +\infty$,
\item[] for $\alpha=1$,
$\disp e^{-\alpha \,\beta} \, (e^\beta- e^{-\beta})\to 1$,
\end{description}
these three limits hold as $\beta$ goes to $+\8$.
From this, we get that for the three cases of possible values of  $\alpha$,
the corresponding limits for $(b-c)$ are the same:
\begin{description}
	\item[] for $\alpha>1$,
$b-c\to 0.
$
\item[] for $\alpha<1$,
$b-c\to +\infty$,
\item[] for $\alpha=1$,
$b-c\to 1$.
\end{description}

Finally, from this we get that for $\alpha>1$,
$$
 \lim_{\beta\to+\8}e^{\beta}\,  \frac{H(\0) }{H(\1)}= 1,
$$

 for $\alpha=1$,
$$
 \lim_{\beta\to+\8}e^{\beta}\,  \frac{H(\0) }{H(\1)}= \frac{1+ \sqrt{5}}{2},
$$
and for $0<\alpha<1$,
$$
 \lim_{\beta\to+\8}e^{\beta}\,  \frac{H(\0) }{H(\1)}=+ \infty .
$$
\end{proof}

\section{End of the proof of the Theorem}\label{sec-end-proof}
Now, we can finish the proof of our Main Theorem.
We recall that any accumulation point for $\mu_\be$ is a $A$-maximizing measure.
Hence, such an
accumulation point is a convex combination of the two Dirac measures
$\delta_{\0}$ and $\delta_{\1}$.
This convex combination can be found if we get an estimate for
$\disp\lim_{\be\to+\8}\frac{\mu([0])}{\mu([1])}$.
We get
\begin{eqnarray}
\frac{\mu([0])}{\mu([1])}&=&\frac{\sum_{n=1}^{+\8}\mu([0^n*_0])}{\sum_{n=1}^{+\8}\mu([1^n*_1])}\nonumber\\
&=&\frac{\sum_{n=1}^{+\8}H(0^n*_0)\nu([0^n*_0])}{\sum_{n=1}^{+\8}H(1^n*_1)\nu([1^n*_1])}\nonumber\\
&=&\frac{\sum_{n=1}^{+\8}H(0^n*_0)e^{-(n-1)\,P- \beta \,
(\frac{1}{2^2}+...+ \frac{1}{2^n})}}{\sum_{n=1}^{+\8}H(1^n*_1)e^{-(n-1)\,P- 3\beta \,
(\frac{1}{2^2}+...+ \frac{1}{2^n})}}\,\frac{\nu([0*_0])}{\nu([1*_1])}\nonumber\\
&=&\frac{\sum_{n=1}^{+\8}e^{\be(1-\frac1{2^{n-1}})}\frac{H(0^n*_0)}{H(1^n*_1)}\mu([1^n*_1])}{\sum_{n=1}^{+\8}\mu([1^n*_1])}\,\frac{\nu([0*_0])}{\nu([1*_1])}.
\label{eq1-serie-cv-rapport}
\end{eqnarray}

The proof will follow from the next technical lemma:
\begin{lemma}\label{lem-limratio-H0nH1n}
There exists $\beta_0$ such that for every $n\ge 3$, for every $\be\ge \beta_0$ and for every $\al$
$$\left|e^{\be(1-\frac1{2^{n-1}})}\frac{H(0^n*_0)}{H(1^n*_1)}\times \frac1{e^{\be}\frac{H(\0)}{H(\1)}}-1\right|\le e^{-\frac\be8}.$$
\end{lemma}
\begin{proof}
We re-employ notations from the proof of Proposition \ref{prop-cv-H0H1}.
We denote by $R_{n-1}(1)$ the tail
$$
 R_{n-1}(1)=F(P,\beta) - F_{n-2}(P, \beta)=
\sum_{k=n-1}^\infty
e^{-k\, P\ + \frac{\beta}{2^{k+1}}},
$$

$R_{n-1}(3)$ the tail
$$
 R_{n-1}(3)=F(P,3\,\beta) - F_{n-2}(P,3\, \beta)=
\sum_{k=n-1}^\infty e^{-k\, P\ +
\frac{3\beta}{2^{k+1}}}
$$

and
$$
 \Delta_{n-1}=R_{n-1}(1)-R_{n-1}(3)
= e^{-(n-1) \, P}\,(
e^{\frac{\beta}{2^{n}}}-
e^{\frac{3\,\beta}{2^{n}}  }   )+...\,\,.
$$
Then,
\begin{eqnarray}
e^{\beta-\frac{\beta}{2^{n-1} }}\,\frac{H(0^n*_0) }{H(1^n*_1)}&=&
 \frac{ a + b x + x \Delta_{n-1} e^{-2 \beta}
 (1 + e^{-P-\alpha\, \beta} )+ x R_{n-1} (3)
e^{- 2\beta}(1 + e^{-P-\alpha\, \beta}  )}
{c+dx +  R_{n-1} (3)
e^{- 2\beta}(1 + e^{-P-\alpha\, \beta}  )}\nonumber\\
&=&
\label{eq3}
x + \frac{x \Delta_{n-1} e^{-2 \beta}
(1 + e^{-P-\alpha\, \beta} )}{c+dx +  R_{n-1} (3)\,
e^{- 2\beta}(1 + e^{-P-\alpha\, \beta}  )}.
\end{eqnarray}
Remember that by definition we have $\disp x=e^{\be}\frac{H(\0)}{H(\1)}$.
Now Equation \eqref{equ-H1n} yields
$$\frac{H(1^n*_1)}{H(\1)}\frac{e^P+e^{-\al\be}}{(e^P-1)}e^{-(n-1)P+\frac{3\be}{2^n}-2\be}=dx+c+  R_{n-1} (3)
e^{- 2\beta}(1 + e^{-P-\alpha\, \beta}  ).$$
If $n$ goes to $+\8$ the right hand side term of this equality goes to $c+dx$. On the other side it is always non-negative. This shows that $c+dx$ is always non-negative.
Therefore \eqref{eq3} yields
$$\left|e^{\be(1-\frac1{2^{n-1}})}\frac{H(0^n*_0)}{H(1^n*_1)}\times \frac1{e^{\be}\frac{H(\0)}{H(\1)}}-1\right|\le \frac{\left|\Delta_{n-1}\right|}{R_{n-1}(3)}.$$
Now, note that $R_{n-1}(1)=F(P,\frac\beta{2^{n-1}})$ and $R_{n-1}(3)=F(P,\frac{3\beta}{2^{n-1}})$. Then, Lemma \ref{lem-FZbeta} shows that $\disp \frac{\left|\Delta_{n-1}\right|}{R_{n-1}(3)}$ is of order $\disp P(\beta)\frac\beta{2^n}e^{\frac{3\beta}{2^{n-1}}}$. Remember that $P$ converges to 0 at least in $e^{-\beta}$.
For $n\ge 3$ and for $\beta$ sufficiently big, $\disp P(\beta)\frac\beta{2^n}e^{\frac{3\beta}{2^{n-1}}}$ is lower than $e^{-\frac\be8}$.
\end{proof}

Now Equation \eqref{eq1-serie-cv-rapport} and Lemma \ref{lem-limratio-H0nH1n} show that we get for every $\beta\ge \beta_0$
$$e^{\be}\frac{H(\0)}{H(\1)}(1-e^{-\frac\be8})\frac{\nu([0*_0])}{\nu([1*1])}\le
\frac{\mu([0])}{\mu([1])}\le e^{\be}\frac{H(\0)}{H(\1)}(1+e^{-\frac\be8})\frac{\nu([0*_0])}{\nu([1*1])},$$
(for $\beta$ big the terms $\mu([0^k*_0])$ and $\mu([1^k*_1])$, $k=1,2$ are very small since $\mu_\beta$ ``goes'' to a combination of $\delta_{\0}$ and $\delta_{\1}$). Then Corollary \ref{coro-lim-ratio-estrella} and Proposition \ref{prop-cv-H0H1} conclude the proof.

\end{document}